\documentclass{amsart}

\usepackage{amscd}
\usepackage{amsmath, amssymb}
\usepackage{amsfonts}
\usepackage{color}

\newcommand{\diam}{\mathrm{diam}}

\renewcommand{\leq}{\leqslant}
\renewcommand{\geq}{\geqslant}

\newcommand{\be}{\begin{equation}}
\newcommand{\ee}{\end{equation}}

\newcommand{\p}{\partial}
\newcommand{\ol}{\overline}
\newcommand{\ul}{\underline}

\begin{document}
\newtheorem{claim}{Claim}
\newtheorem{theorem}{Theorem}[section]
\newtheorem{lemma}[theorem]{Lemma}
\newtheorem{corollary}[theorem]{Corollary}
\newtheorem{proposition}[theorem]{Proposition}
\newtheorem{question}{question}[section]
\newtheorem{definition}[theorem]{Definition}
\newtheorem{remark}[theorem]{Remark}

\numberwithin{equation}{section}

\title[Prescribed curvature equations]
{The Dirichlet problem for a class of curvature equations in Minkowski space}

\author[M. Guo]{Mengru Guo}
\address{School of Mathematics, Harbin Institute of Technology,
	Harbin, Heilongjiang, 150001, China}
\email{22B912007@stu.hit.edu.cn}

\author[H. Jiao]{Heming Jiao}
\address{School of Mathematics and Institute for Advanced Study in Mathematics, Harbin Institute of Technology,
         Harbin, Heilongjiang, 150001, China}
\email{jiao@hit.edu.cn}
\thanks{The second author is supported by the National Natural Science Foundation of China (Grant No. 12271126),
the Natural Science Foundation of Heilongjiang Province (Grant No. YQ2022A006),
and the Fundamental Research Funds for the Central Universities (Grant No. HIT.OCEF.2022030).}

\begin{abstract}
In this paper, we study the Dirichlet problem for a class of prescribed curvature equations in Minkowski space.
We prove the existence of smooth spacelike hypersurfaces
with a class of prescribed curvature and general boundary data based on establishing the \emph{a priori} $C^2$ estimates.

\end{abstract}

\maketitle

\section{Introduction}

Let $\mathbb{R}^{n,1}$ be the Minkowski space, i.e., the space $\mathbb{R}^n \times \mathbb{R}$ equipped with the metric
\[
ds^2=dx^2_1+\cdots+dx^2_n-dx^2_{n+1}.
\]
In \cite{JaoSun22}, the second author and Sun studied a class of curvature equations with constant boundary condition
in a domain of Euclidean space.
The purpose of the current work is to extend the results of \cite{JaoSun22} to the Minkowski case with general boundary data.
Unlike the Euclidean case, bad terms including the square of curvatures appear
when we differentiate the equations twice in the Minkowski context.
That is why it is still an open problem whether the Dirichlet problem for $k$-curvature equations
\begin{equation}
\label{sigmak}
\sigma_{k} (\kappa [M_u]) = \psi
\end{equation}
are solvable
for $3\leq k \leq n-3$, although it is well known long ago for corresponding equations in the Euclidean context,
where $M_u=\{(x,u(x)): x\in\Omega\}$ is the graphic hypersurface defined by the function $u$ and
\[
\sigma_{k} (\kappa) = \sum_ {1 \leq i_{1} < \cdots < i_{k} \leq n}
\kappa_{i_{1}} \cdots \kappa_{i_{k}}
\]
are the $k$-th elementary symmetric functions, $k=1, \ldots, n$. The reader is referred to \cite{CNSV, Ivochkina90, Ivochkina91, ILT96} for
the study of the Dirichlet problem for \eqref{sigmak} and more general curvature equations in the ambient Euclidean space.

In this paper, we are concerned with the graph of a spacelike function $u$ defined in a bounded domain $\Omega \subset \mathbb{R}^n$.
Here the terminology ``spacelike'' means that
\[
\sup_{\ol{\Omega}}|Du|<1.
\]
It is easy to find that for a spacelike function $u$, the Minkowski metric restricted to $M_u$ defines a Riemannian metric on $M_u$:
\[
g_{ij} = \delta_{ij} - D_i u D_j u, \ 1\leq i,j\leq n.
\]
The second fundamental form of $M_u$ with respect to its upward unit normal vector field
\[\nu=\frac{(Du, 1)}{\sqrt{1-|Du|^2}}\]
is given by
\[h_{ij}=\frac{u_{ij}}{\sqrt{1-|Du|^2}}.\]
The principal curvatures $\kappa_1, \ldots, \kappa_n$ of $M_u$ are defined by the eigenvalues of
$\{h_{ij}\}$ with respect to $\{g_{ij}\}$.

A $C^2$ regular spacelike hypersurface $M_u$ is called $(\eta, n)$-convex if the principal curvatures
$\kappa = (\kappa_1, \ldots, \kappa_n) \in \Gamma$ at each $X \in M_u$, where $\Gamma$ is the symmetric cone defined by
\[
 \Gamma := \{\kappa = (\kappa_1, \ldots, \kappa_n) \in \mathbb{R}^n: \lambda_i = \sum_{j \neq i} \kappa_i > 0, i = 1, \ldots, n\}.
\]
In addition, we call a $C^2$ spacelike function $u: \Omega \rightarrow \mathbb{R}$ admissible if its graph
$M_u$ is $(\eta, n)$-convex.
Such hypersurface was introduced by Sha \cite{Sha86} and Wu \cite{Wu87} to describe the boundaries of Riemannian manifolds
which have the homotopy type of a CW-complex and was studied extensively in \cite{Sha87, HL13}.

Let $H (X)$ be the mean curvature of $M_u$ at $X \in M_u$. Define the $(0, 2)$-tensor field $\eta$ on $M$ by
\[
\eta = H g - h.
\]
It is clear that a hypersurface $M_u$ is $(\eta, n)$-convex if and only if $\eta$ is positive definite at each point of $M_u$.
The $(\eta, n)$-curvature at $X \in M_u$ is defined by $K_\eta (X) := \lambda_1 (X) \cdots \lambda_n (X)$, where
\[
\lambda_i (X) = \sum_{j\neq i} \kappa_j (X).
\]
Obviously, we have
\[
K_\eta (X) = \det (g^{-1} \eta (X)).
\]
In this paper, we consider the existence of smooth spacelike $(\eta, n)$-convex hypersurface satisfying the Dirichlet problem
\begin{equation}
	\label{jg--1}
	\left\{ \begin{aligned}
		K_\eta [M_u] & = \psi (x, u) & \;\; \mbox{ in } \Omega, \\
		u &= \varphi & \;\;~  \mbox{ on } \partial \Omega,
	\end{aligned} \right.
\end{equation}
where $\psi \in C^\infty (\ol \Omega \times \mathbb{R}) > 0$ and $\varphi \in C^\infty (\partial \Omega)$.
We usually need some geometric conditions on $\Omega$ when we consider Dirichlet boundary value problem.
A bounded domain $\Omega \subset \mathbb{R}^n$ is called admissible if
there exists a positive constant $K$ such that for each $x \in \partial \Omega$,
\[
(\kappa^b_1 (x), \ldots, \kappa^b_{n - 1} (x), K) \in \Gamma,
\]
where $\kappa^b_1 (x), \ldots, \kappa^b_{n - 1} (x)$ are the principal curvatures of $\partial \Omega$ at $x$.
Using almost the same arguments of Lemma 2.1 in \cite{Bayard03}, we can prove that
Every affine spacelike data admits an admissible extension to $\Omega$ if and only if $\Omega$
is convex and admissible.
\begin{theorem}
\label{js-thm1}
Assume that $\Omega$ is convex and admissible with smooth boundary $\partial \Omega$.
Suppose
$\psi = \psi (x,z) \in C^\infty (\ol \Omega \times \mathbb{R}) > 0$, $\psi_z \geq 0$ and $\varphi \in C^\infty (\partial \Omega)$.
In addition, assume that there exists an admissible subsolution $\underline{u} \in C^{2} (\overline{\Omega})$ satisfying
\begin{equation}
\label{subsol}
\left\{ \begin{aligned}
   K_\eta [M_{\underline{u}}] & \geq \psi (x, \ul u)  & \;\; \mbox{ in } \Omega,\\ 
                 \ul u &= \varphi  & \;\;~ \mbox{ on } \partial \Omega.
\end{aligned} \right.
\end{equation}
Then there exists a unique admissible solution $u \in C^{\infty} (\ol \Omega)$ to \eqref{jg--1}.
\end{theorem}
A necessary condition for \eqref{jg--1} being solvable is that $|\partial \varphi| < 1$ on $\partial \Omega$, where $\partial \varphi$ means the
gradient of $\varphi$ on the boundary $\partial \Omega$. It is also of interest to construct subsolutions satisfying \eqref{subsol}.
Using methods in Section 2 of \cite{CNSIII}, we may construct a function $v$ satisfying $\lambda (D^2 v) \in \Gamma$ on $\ol \Omega$,
$\sup_{\ol \Omega}|Dv| < 1$ and $v \big|_{\partial \Omega} \equiv \varphi$ in an admissible domain provided $\varphi$ can be extended to a convex spacelike
function on $\ol \Omega$. Thus, the function $v$ can be a subsolution to \eqref{jg--1} if $\sup\psi$ is sufficiently small. In particular,
we can construct a subsolution for spacelike affine $\varphi$ in this case. If $\Omega$ is strictly convex and $\varphi$ can be extended to
a spacelike strictly convex function on $\ol\Omega$, the strictly convex solution to the Lorentz-Gaussian curvature equation
\begin{equation}
\label{Lorentz-Gauss}
\left\{ \begin{aligned}
  \frac{\det(D^2 u)}{(1-|Du|^2)^{(n+2)/2}} & = \psi (x, u)  & \;\; \mbox{ in } \Omega,\\ 
                u &= \varphi  & \;\;~ \mbox{ on } \partial \Omega
\end{aligned} \right.
\end{equation}
can be a subsolution to \eqref{jg--1}. The reader is referred to \cite{Delano90} for the solvability of \eqref{Lorentz-Gauss}.
Therefore, we have
\begin{theorem}
\label{jg-thm1}
Let $\Omega$ be a bounded and strictly convex domain in $\mathbb{R}^n$ with $\partial \Omega \in C^{\infty}$.
Suppose
$\psi = \psi (x,z) \in C^\infty (\ol \Omega \times \mathbb{R}) > 0$ and $\psi_z \geq 0$.
Assume that $\varphi \in C^\infty (\ol \Omega)$ is spacelike and strictly convex.
Then there exists a unique admissible solution $u \in C^{\infty} (\ol \Omega)$ to \eqref{jg--1}.
\end{theorem}

The Dirichlet problem for prescribed mean curvature equation in Minkowski space (\eqref{sigmak} with $k=1$)
was first solved by Bartnik and Simon \cite{Bartnik82},
while the prescribed Lorentz-Gauss curvature equation \eqref{Lorentz-Gauss} was studied by Delano\`e \cite{Delano90} (see \cite{Guan98} also).
The Dirichlet problem for prescribed scalar curvature equation (\eqref{sigmak} with $k=2$) was solved by Bayard \cite{Bayard03}
with the dimension $n = 3, 4$ and by Urbas \cite{Urbas03} with general $n$. Schn\"{u}rer \cite{Sch02} considered a class of prescribed curvature
equations which exclude \eqref{sigmak} with $1<k<n$ and \eqref{jg--1}. There are also interesting works on curvature equations in Lorentzian manifolds
without boundary (see \cite{Gerhardt01, Gerhardt03}). Ren-Wang's methods in \cite{RW19, RW20} can apply to equation \eqref{sigmak} to solve the cases
$k=n-1$ and $k=n-2$. However, the solvability of \eqref{sigmak} is still an open problem for $3\leq k \leq n-3$. The reason is the lack of the
\emph{a priori} $C^2$ estimates. Huang \cite{Huang 13} established the second estimates for \eqref{sigmak} under an additional condition.
Recently, Wang-Xiao \cite{WX22} and Ren-Wang-Xiao \cite{RWX24} considered the Entire spacelike hypersurfaces with prescribed curvatures in Minkowski space.

Curvature equations in Euclidean context were extensively studied by various authors, we refer the reader to
\cite{CNSV, GS04, GRW15, Ivochkina90, Ivochkina91, ILT96, JW21, JaoSun22} and the references therein for related works.
In particular, \eqref{jg--1} in the Euclidean context with $\varphi \equiv$ constant was considered in \cite{JaoSun22}.

In this paper, we establish a Pogorelov type estimate for second order derivatives (Theorem \ref{gj-thm1}) where we have used a method from
\cite{Guan14} and \cite{Urbas02} to deal with bad third order terms. When the boundary data $\varphi$ is not constant, the estimates for 
double normal derivatives on the boundary become much more complicated. We shall use an idea of  \cite{Ivochkina91} and \cite{ILT96} to
overcome these difficulties.

The rest of this paper is organized as follows. In Section 2, we provided
some preliminaries. In Section 3, The $C^1$ estimates are established.
Section 4 and 5 are devoted to the global and boundary estimates for second order
derivatives respectively.



\section{Preliminaries}

In this work, $\varphi_i = \frac{\p \varphi}{\p x_i}$, $\varphi_{ij} = \frac{\p^2 \varphi}{\p x_i \p x_j}$,
$D\varphi=(\varphi_{1}, \cdots, \varphi_{n})$ and $D^2 \varphi = (\varphi_{ij})$ represent ordinary first-order and second-order derivatives, gradient and Hessian matrix of
a function $\varphi \in C^2 (\Omega)$ respectively.

Let $u$ be a spacelike function and $M_u$ its graph.
Let
$\epsilon_{n+1} = (0, \cdots, 0, 1) \in \mathbb{R}^{n+1}$. Thus,
the height function of $M_u$,
$u(x)=-\langle X, \epsilon_{n+1}\rangle$, where $\langle\cdot,\cdot\rangle$ is the Minkowski inner product.

We see that the principle curvatures of $M_u$ are eigenvalues of the matrix
\[
\frac{1}{w} \left(I + \frac{Du \otimes Du}{w^2}\right) D^2 u
\]
or the symmetric matrix
$A [u] = \{a_{ij}\}$:
\begin{equation}
\label{matrix}
a_{ij}=\frac{1}{w}\gamma^{ik}u_{kl}\gamma^{lj},
\end{equation}
where $\gamma^{ik}=\delta_{ik}+\frac{u_iu_k}{w(1+w)}$ and $w=\sqrt{1-|Du|^2}.$
Note that $\{\gamma^{ij}\}$ is invertible and its inverse is the square root of $\{g_{ij}\}$:
$\{\gamma_{ij}\}=\{\delta_{ij}-\frac{u_iu_j}{1+w}\}$.

For $r \in S^{n \times n}$ and $p \in \mathbb{R}^n$ with $|p| < 1$, define
\[
\lambda (r, p) = \lambda \left(\left(I + \frac{p \otimes p}{1 - |p|^2}\right) r\right)
\]
and
\[
S_k (r, p) = \sigma_k (\lambda (r,p)).
\]
Following \cite{ILT96}, we introduce the following notations. For $p \in \mathbb{R}^n$, $i = 1, \ldots, n$, let
$p(i)$ be the vector obtained by setting $p_i = 0$, $r (i)$ the matrix obtained by setting the $i^{th}$ row
and column to zero and $r (i,i)$ represent the matrix obtained by setting $r_{ii} = 0$. Denote
\[
S_{k,i} (r,p) = S_k (r(i), p(i)).
\]
Similar calculations as in \cite{ILT96} yield
\begin{equation}
	\label{ILT}
	S_k (r, p) = \frac{1 - |p(i)|^2}{1-|p|^2} r_{ii} S_{k-1; i} (r, p) + O (|r (i, i)|^k)
\end{equation}
for all $1\leq i \leq n$ and $1\leq k \leq n$, where $S_0$ is defined by $S_0 \equiv 1$.
More general, let $\tilde{e}_1, \ldots, \tilde{e}_n$ be local orthonormal frames in $\mathbb{R}^n$ and
write $e_i (x) = e^j_i (x)\partial_j$ for $i = 1, \ldots n$, where $\partial_1,\cdots,\partial_n$ is the rectangular coordinate system
on $\mathbb{R}^n$. Denote
\[
\tilde{\nabla}_i u = e_i^ju_j, \tilde{\nabla}_{ij} u = e_i^k e_j^l u_{kl}
\mbox{ and }
\tilde{\nabla}^2 u = \{\tilde{\nabla}_{ij} u\}.
\]
We have
\begin{equation}
	\label{ILT-1}
\begin{aligned}
	S_k (D^2 u, Du) = \,& S_k (\tilde{\nabla}^2 u, \tilde{\nabla} u) \\
  = \,& \frac{1 - |\tilde{\nabla} u(i)|^2}{1-|\tilde{\nabla} u|^2}
         \tilde{\nabla}_{ii} u S_{k-1; i} (\tilde{\nabla}^2 u, \tilde{\nabla} u) + O (|\tilde{\nabla}^2 u (i, i)|^k).
\end{aligned}
\end{equation}
When $n=2$, equation \eqref{jg--1} is the classic prescribed Gauss curvature equation. For general $n$,
\begin{equation}
	\label{jg-3}
	K_\eta [M] = K_\eta (\kappa) = \sum_{i=2}^n \sigma_1 (\kappa)^{n-i} \sigma_i (\kappa).
\end{equation}
By \eqref{jg-3} and \eqref{ILT-1}, we have
\begin{equation}
\label{jg-4}
\begin{aligned}
K_\eta (M_u) = \,& \frac{1}{w^n}\sum_{i=2}^n S_1^{n-i} (\tilde{\nabla}^2 u, \tilde{\nabla} u) S_i (\tilde{\nabla}^2 u, \tilde{\nabla} u)\\
  = \,& \frac{1}{w^n} \left(\frac{1 - |\tilde{\nabla} u(n)|^2}{1 - |\tilde{\nabla} u|^2}\right)^{n-1}  S_{1;n} (\tilde{\nabla}^2 u, \tilde{\nabla} u) (\tilde{\nabla}_{nn} u)^{n-1}\\
   & + \sum_{i=1}^{n-2} P_i (\tilde{\nabla}_{nn} u)^{i} + P_0,
\end{aligned}
\end{equation}
where $P_i$ depend only on $\tilde{\nabla}_{\alpha \beta} u$ ($\alpha + \beta < 2n$) and $\tilde{\nabla} u$, $i=0, 1, \ldots n-2$.

For $\kappa \in \Gamma$, let
\begin{equation}
\label{lambda}
\lambda_i := \sum_{j \neq i} \kappa_j, i = 1, \ldots, n.
\end{equation}
and
\begin{equation}
\label{def-h}
f (\kappa) := \lambda_1 \cdots \lambda_n.
\end{equation}
Thus, we find
\[
K_\eta [M_u] = f (\kappa),
\]
where $\kappa = (\kappa_1, \ldots, \kappa_n)$ are the principal curvatures of $M_u$.
We need some basic properties of $f (\kappa)$ (seeing \cite{JL20}).
First there exists an positive constant $\delta_0$ depending only on $n$ such that
\begin{equation}
\label{cj-10}
f_i (\kappa) = \frac{\partial f (\kappa)}{\partial \kappa_i} > 0, \mbox{ in } \Gamma, i = 1, \ldots, n,
\end{equation}
\begin{equation}
\label{cj-9}
f^{1/n} (\kappa) \mbox{ is concave in } \Gamma,
\end{equation}
\begin{equation}
\label{js-3}
f > 0 \mbox{ in } \Gamma \mbox{ and } f = 0 \mbox{ on } \partial \Gamma,
\end{equation}
and
\begin{equation}
\label{cj-6}
f_j (\kappa) \geq \delta_0 \sum_i f_i (\kappa), \mbox{ if } \kappa_j < 0, \forall \kappa \in \Gamma.
\end{equation}
In addition, for any constant $A > 0$ and any compact set $K$ in $\Gamma$
there is a number $R = R (A, K)$ such that
\begin{equation}
\label{cj-2}
f (\kappa_1, \ldots, \kappa_{n-1}, \kappa_n + R) \geq A, \mbox{ for all } \kappa \in K.
\end{equation}
 In this paper, we denote $\sigma_{m;i_1,\cdots,i_k}(\kappa)=\sigma_m(\kappa)|_{\kappa_{i_{1}}=\cdots=\kappa_{i_{k}}=0}$ for integer
 $1\leq i_1,\cdots,i_k\leq n, 1\leq m\leq n$
 and
 $n-k\leq m$.
Obviously, we have
 \[
 \frac{\partial\sigma_{k}}{\partial\kappa_{i}}(\kappa)=\sigma_{k-1;i}(\kappa),
 \]
\[
\sum_i \sigma_{k-1; i}(\kappa) = (n-k+1) \sigma_{k-1}(\kappa)\]
and
\[
\sum_i \sigma_{k-1; i} (\kappa) \kappa_i = k \sigma_k (\kappa).
\]

Choose local orthonormal
frames $\{e_1,e_2,\cdots,e_n\}$ on $TM_u$. $\nabla$ denotes the induced Levi-Civita connection on $M$. For a function $v$ on $M_u$, we denote $\nabla_i v=\nabla_{e_i}v,$
$\nabla_{ij} v = \nabla^2 v (e_i, e_j),$ etc in this paper.
Thus, we have
\[|\nabla u|=\sqrt{g^{ij}u_{i}u_{j}}=\frac{|Du|}{\sqrt{1-|Du|^2}}.\]

In normal coordinates, we have the following fundamental formulae and equations for the hypersurface $M$ in Minkowski space $\mathbb{R}^{n, 1}$:
\begin{equation}\label{Gauss}
\begin{aligned}
\nabla_{ij} X = \,& h_{ij}\nu \quad {\rm (Gauss\ formula)}\\
\nabla_i \nu= \,& h_{ij} e_j \quad {\rm (Weigarten\ formula)}\\
\nabla_k h_{ij} = \,& \nabla_j h_{ik} \quad {\rm (Codazzi\ equation)}\\
R_{ijst} = \,& -(h_{is}h_{jt}-h_{it}h_{js})\quad {\rm (Gauss\ equation)},
\end{aligned}
\end{equation}
where $h_{ij} = \langle D_{e_i}\nu,  e_j\rangle$ is the second fundamental form of $M$.

\section{$C^1$ estimates}

In this section, we establish the $C^1$ estimates for the admissible solution to \eqref{jg--1}. Indeed,
we prove there exists a constant $0 < \theta_0 < 1$ such that
\begin{equation}
	\label{spacelike}
	\sup_{\ol \Omega} |Du|\leq 1-\theta_0.
\end{equation}
First, by the arithmetic and geometric mean inequality, we see
\[
\left(K_\eta [M_u]\right)^{1/n} = \sqrt[n]{\lambda_1 \cdots \lambda_n} \leq \frac{\lambda_1 + \cdots + \lambda_n}{n} = \frac{n-1}{n} H,
\]
where $\lambda_i = \sum_{j\neq i}\kappa_j$, $i=1, \ldots, n$, $\kappa_1, \ldots, \kappa_n$ are the principal curvatures of $M_u$
and $H := \sigma_1 (\kappa)$ is the mean curvature of $M_u$.
By \cite{Bartnik82}, there exists a spacelike solution $\ol u \in C^\infty (\ol \Omega)$ to the Dirichlet problem
\[
	\left\{ \begin{aligned}
		H [M_{\ol u}] & = \frac{n}{n-1} \psi^{1/n} (x, \ol u) & \;\; \mbox{ in } \Omega, \\
		\ol u &= \varphi & \;\;~  \mbox{ on } \partial \Omega.
	\end{aligned} \right.
\]
Thus, by the maximum principle, we have
$$ \ul u \leq u \leq \ol u \text{ in } \Omega \text{ and } \ol u = u = \ul u \text{ on } \partial{\Omega},$$
which implies
\[
 \frac{\partial \ul u}{\partial \gamma} \leq \frac{\partial u}{\partial \gamma} \leq \frac{\partial \ol u}{\partial \gamma}.
\]
where $\gamma$ is the interior unit normal to $\partial \Omega$.
Then we have
\begin{equation}
\label{jg-7'}
\sup_{\ol \Omega} |u| \leq \max\{\sup_{\ol \Omega} |\ol u|, \sup_{\ol \Omega} |\ul u|\}
\end{equation}
and
\begin{equation}
\label{jg-7}
\sup_{\partial \Omega} |Du| \leq \max\{\sup_{\partial \Omega} |D \ul u|, \sup_{\partial \Omega} |D \ol u|\} \leq 1 - \theta
\end{equation}
for some constant $0 < \theta < 1$ since $\ul u$ and $\ol u$ are both spacelike.
%
Next we prove an upper bound for
$$ \tilde{w} =\frac{1}{\sqrt{1-|Du|^2}}=\frac{1}{w}.$$
\begin{theorem}
\label{gradient}
Let $u \in C^3 (\Omega) \cap C^1 (\ol \Omega)$ be an admissible solution
of \eqref{jg--1}. Suppose the smooth function $\psi$ satisfies $\psi_z \geq 0$.
Then
\begin{equation}
\label{gradient-1}
\sup_{\ol \Omega} \tilde{w}
\leq e^{\left(\frac{\sup_{\ol \Omega}|D\psi|}{n \inf_{\ol{\Omega}}\psi}(2 \sup_{\partial \Omega} |\varphi|) + \diam (\Omega) \right)}
\sup_{\partial \Omega} \tilde{w}.
\end{equation}
\end{theorem}
\begin{proof}
Let
\[
\hat{Q} := \tilde{w} e^{Bu},
\]
where 
$B$ is a positive constant to be determined later.
Suppose the maximum value of $\hat{Q}$ is achieved at an interior point $x_0 \in \Omega$. It follows that
\[
Q := \log \hat{Q} = \log \tilde{w} + Bu
\]
also attains its maximum at $x_0$. Let $\epsilon_1, \ldots, \epsilon_{n+1}$ be a standard basis of $\mathbb{R}^{n+1}$. We may assume
$u_1 (x_0) = |Du (x_0)|$ and $u_j (x_0) = 0$ for $j \geq 2$ by a rotation of $\epsilon_1, \ldots, \epsilon_{n}$ if necessary.
Let $\{e_1,e_2,\cdots,e_n\}$ be an orthonormal frame on $M_u$ around $X_0 = (x_0, u (x_0))$ such that, at $x_0$,
$$\nabla_1u=\frac{|Du|}{w}=|\nabla u|, \nabla_i u=u_i=0, \text{ for } i\geq 2.$$
We may set
\[
e_i = \gamma^{is} \tilde{\partial}_s, \ i =1, \ldots,n,
\]
where $\gamma^{is} := \delta_{is}+\frac{u_iu_s}{w(1+w)}$ and $\tilde{\partial}_s := \epsilon_s + u_s \epsilon_{n+1}$.
We may further rotate $\epsilon_2, \ldots, \epsilon_n$ such that $\{u_{ij}\}_{i,j \geq 2}$ is
diagonal at $x_0$.
By the Weingarten formula, we have
$$
\nabla_i \tilde{w} = -\nabla_i \langle \nu, \epsilon_{n+1} \rangle = -\langle h_{ij} e_j, \epsilon_{n+1} \rangle
  = - h_{ij} \langle \gamma^{js} \tilde {\partial_s}, \epsilon_{n+1}\rangle = h_{ij} \nabla_j u.
$$
At $x_0$ where $Q$ attains its maximum, we have
\begin{equation}
\label{gj-1}
0 = \nabla_i Q = \frac{\nabla _i \tilde{w}}{\tilde{w}} +B \nabla_i u
= \frac{h_{i1} \nabla_1 u}{\tilde{w}} + B \nabla_i u
\end{equation}
and
\begin{equation}
\label{gj-2}
	0 \geq \nabla_{ii} Q =
	 \frac{\nabla_i h_{i1} \nabla_1 u + h_{il} \nabla_{il} u}{\tilde{w}}
      - \frac{\left(h_{i1} \nabla_1 u \right)^2}{\tilde{w}^2} + B \nabla_{ii} u.
\end{equation}
We may assume $D u(x_0) \neq 0$ for otherwise we are done. In the rest of the proof, all the calculations are
carried out at $X_0$.
First we note that
\[
h_{11} = -B \tilde{w} \mbox{ and } h_{1i} = 0 \mbox{ for } i \geq 2
\]
by \eqref{gj-1} and the fact $\nabla_i u (x_0) = 0$ for $i \geq 2$. Since at $X_0$,
\[
h_{11} = \frac{u_{11}}{w^3}, h_{1i} = \frac{u_{1i}}{w^2} \mbox{ and }
h_{ij} = \frac{u_{ij}}{w}
  \mbox{ for } i, j \geq 2,
\]
we find that the matrix $\{h_{ij}\}$ is diagonal at $X_0$, and so is $\{F^{ij}\}$,
where
\[
F^{ij} := \frac{\partial f (\lambda (h))}{\partial h_{ij}}.
\]
Next, by the Codazzi equation and differentiating the equation \eqref{jg--1}, we have
\begin{equation}
\label{gj-3}
\begin{aligned}
F^{ii} \nabla_i h_{i1} = \,& F^{ii} \nabla_1 h_{ii} = \nabla_1 \psi
 = \psi_{x_j} \nabla_1 x_j + \psi_u \nabla_1 u \\
		= \,& \frac{\psi_{x_1}}{w} + \psi_u \nabla_1 u
		= \tilde{w} (\psi_{x_1} +\psi_u u_1 )
		\geq \tilde{w} {\psi_{x_1}},
\end{aligned}
\end{equation}
where the last inequality is due to that $\psi_u \geq 0$.
By the Gauss formula, we have
\begin{equation}
\label{gj-4}
\nabla_{ij} u=-\nabla_{ij}\langle X,\epsilon_{n+1}\rangle=-\langle \nabla_{ij}X,\epsilon_{n+1}\rangle =
-h_{ij}\langle \nu,\epsilon_{n+1}\rangle = \tilde{w} h_{ij}.
\end{equation}
Combining \eqref{gj-1}-\eqref{gj-4}, we obtain
\begin{equation}
\label{gj-5}
\begin{aligned}
		0 \geq F^{ii} \nabla_{ii} Q \geq \,& \psi_{x_1} u_1 \tilde{w}
				+ F^{ii} h_{ii}^2
		- B^2 F^{11} (\nabla_1 u)^2 + B n \psi \tilde{w}\\
  \geq \,& - |D \psi|\cdot |Du| \tilde{w} + B^2 \tilde{w}^2 F^{11} - B^2 F^{11} (\nabla_1 u)^2 + B n \psi \tilde{w}\\
  \geq \,& - |D \psi| \tilde{w} + B^2 F^{11} + B n \psi \tilde{w} \geq (B n \psi - |D \psi|) \tilde{w}.
\end{aligned}
\end{equation}
Then we get a contradiction provided
\[
B > \frac{\sup_{\ol \Omega}|D\psi|}{n \inf_{\ol{\Omega}}\psi}
\]
and
\eqref{gradient-1} follows as \cite{Bayard03}.
\end{proof}
It follows from \eqref{jg-7} and \eqref{gradient-1} that \eqref{spacelike} holds.

\section{Interior and global estimates for second order derivatives}
In this section, we consider the interior and global estimates for second order derivatives.

\begin{theorem}
\label{gj-thm1}
Let $\tilde{\varphi} \in C^2 (\Omega) \cap C^0 (\ol \Omega)$ be a spacelike convex function satisfying
$\tilde{\varphi} > u$ in $\Omega$ and $\tilde{\varphi} = u = \varphi$
on $\partial \Omega$. Then there exist positive constants $\alpha$ and $C$ depending
only on $n$, $\theta_0$ (defined in \eqref{spacelike}) and $\|\psi\|_{C^2}$ such that
\begin{equation}
\label{gj-6}
(\tilde{\varphi} - u)^\alpha |D^2 u| \leq C.
\end{equation}
\end{theorem}
\begin{proof}
Let
\[
W (X, \xi) = \zeta^\alpha e^{\frac{b}{2} |X|^2} h_{\xi \xi}
\]
for $X \in M_u$ and unit $\xi \in T_X M_u$, where $\zeta := \tilde{\varphi} - u$ and $b$ is a positive constant
to be chosen. Suppose the maximum of $W$ is
achieved at $X_0 = (x_0, u (x_0))\in M_u$ and $\xi_0 \in T_{X_0} M_u$.
We choose a  local orthonormal frame $\{e_{1},e_{2},\cdots,e_{n}\}$ about $X_{0}$ such that,
\[
\xi_0 = e_1,\ \nabla_{e_i} e_j = 0 \mbox{ at } X_0.
\]
We may also assume that $\{h_{ij}\}$ is diagonal at $X_0$ and furthermore,
\[
h_{11} \geq \cdots \geq h_{nn}.
\]
Let $\eta_{ij} = H\delta_{ij} - h_{ij}$. We find $\{\eta_{ij}\}$ is also diagonal at $X_0$ and
\[
\eta_{11} \geq \cdots \geq \eta_{nn}.
\]
In the rest of proof, the calculations are all carried out at $X_0$. First we note that the function (defined near $X_0$)
\[
\alpha \log \zeta + \log h_{11} + \frac{b}{2} |X|^2
\]
achieves its maximum at $X_0$ and therefore,
\begin{equation}
\label{gj-7}
  \alpha \frac{\nabla_i \zeta}{\zeta} + b \langle X, e_i \rangle + \frac{\nabla_i h_{11}}{h_{11}} = 0,
  i = 1, \ldots, n
\end{equation}
and
\begin{equation}
\label{gj-8}
0 \geq \alpha\frac{\nabla_{ii} \zeta}{\zeta} - \alpha\left(\frac{\nabla_i \zeta}{\zeta}\right)^2
    + b \big(1 + h_{ii}\langle X, \nu\rangle\big) + \frac{\nabla_{ii} h_{11}}{h_{11}} - \left(\frac{\nabla_i h_{11}}{h_{11}}\right)^2.
\end{equation}
Next, we have
\begin{equation}
\label{gj-9}
\nabla_{ij} \tilde{\varphi} = D^2 \tilde{\varphi} (e_i, e_j) + \sum_{k=1}^n \nu_k \tilde{\varphi}_k h_{ij}.
\end{equation}
Since $\tilde{\varphi}$ is convex, we have
\[
\tilde{F}^{ij} \nabla_{ij} \tilde{\varphi} \geq \sum_{k=1}^n \nu_k \tilde{\varphi}_k \tilde{F}^{ij} h_{ij}
   = \sum_{k=1}^n \nu_k \tilde{\varphi}_k \psi^{1/n},
\]
where
\[
\tilde{F}^{ij} := \frac{\partial f^{1/n} (\lambda (h))}{\partial h_{ij}}.
\]
It follows that
\begin{equation}
\label{gj-10}
\tilde{F}^{ii} \nabla_{ii} \zeta = \tilde{F}^{ii} \nabla_{ii} (\tilde{\varphi} - u) \geq \left(\sum_{k=1}^n \nu_k \tilde{\varphi}_k - \nu_{n+1}\right) \tilde{F}^{ii} h_{ii} \geq - C.
\end{equation}
Using standard formulae and differentiating the equation \eqref{jg--1} twice, we have
\begin{equation}
	\label{GLL-1}
	\begin{aligned}
		\tilde{F}^{ij}\nabla_{ij} h_{ab}=\,&-\tilde{F}^{ij,kl}\nabla_{a}h_{ij}\nabla_{b}h_{kl}-\tilde{F}^{ij}h_{ij}h_{ak}h_{bk}\\
		\,&+\tilde{F}^{ij}h_{ik}h_{jk}h_{ab}+\nabla_{ab}(\psi^{1/n}),
	\end{aligned}
\end{equation}
where
\[
\tilde{F}^{ij,kl} := \frac{\partial^2 f^{1/n} (\lambda (h))}{\partial h_{ij}\partial h_{kl}}.
\]
(The reader is referred to \cite{Urbas 03} for a proof of \eqref{GLL-1}.) It follows that
\begin{equation}
\label{gj-11}
\tilde{F}^{ii} \nabla_{ii}h_{11} = -\tilde{F}^{ij,kl}\nabla_{1}h_{ij}\nabla_{1}h_{kl} - (\psi^{1/n}) h_{11}^2 + h_{11} \tilde{F}^{ii} h_{ii}^2 + \nabla_{11} (\psi^{1/n})
\end{equation}
By \eqref{gj-8}, \eqref{gj-10} and \eqref{gj-11}, we obtain
\begin{equation}
\label{gj-12}
0 \geq - \frac{C\alpha}{\zeta} + b \sum \tilde{F}^{ii} + \tilde{F}^{ii} h_{ii}^2 - C h_{11} + E
\end{equation}
provided $h_{11}$ is sufficiently large, where
\[
E := -\frac{\tilde{F}^{ij,kl}\nabla_{1}h_{ij}\nabla_{1}h_{kl}}{h_{11}} - \alpha \tilde{F}^{ii}\left(\frac{\nabla_i \zeta}{\zeta}\right)^2 - \tilde{F}^{ii} \left(\frac{\nabla_i h_{11}}{h_{11}}\right)^2.
\]
To estimate $E$ we use the following lemma proved by Andrews \cite{Andrews94} and Gerhardt \cite{Gerhardt96} as
in \cite{Guan14}.
\begin{lemma}
\label{AG}
For any symmetric matrix $A = \{a_{ij}\}$, we have
\begin{equation}
\label{AG-1}
\tilde{F}^{ij, kl} a_{ij} a_{kl} = \sum_{i,j} \frac{\partial^2 f^{1/n}}{\partial \kappa_i \partial \kappa_j} a_{ii} a_{jj}
    + \sum_{i\neq j} \frac{(f^{1/n})_i - (f^{1/n})_j}{\kappa_i - \kappa_j} a_{ij}^2.
\end{equation}
\end{lemma}
Set
\[  \begin{aligned}
J \,& = \{i: h_{ii} \leq - s h_{11}\}, \;\;
K = \{i:  h_{ii} > - s h_{11} \}.
  \end{aligned} \]
for fixed $0 < s \leq 1/3$.
We have, by \eqref{AG-1} and the Codazzi equation,
\begin{equation}
\label{gj-S130}
\begin{aligned}
 - \tilde{F}^{ij, kl} \nabla_1 h_{ij} \nabla_1 h_{kl}
\geq \,& \sum_{i \neq j} \frac{\tilde{F}^{ii} - \tilde{F}^{jj}}{h_{jj} - h_{ii}}
           (\nabla_1 h_{ij})^2 \\
 \geq \,& 2 \sum_{i \geq 2} \frac{\tilde{F}^{ii} - \tilde{F}^{11}}{h_{11} - h_{ii}}
            (\nabla_1 h_{i1})^2 \\
 \geq \,& \frac{2}{(1+s) h_{11}} \sum_{i \in K} (\tilde{F}^{ii} - \tilde{F}^{11})
            (\nabla_i h_{11})^2.
\end{aligned}
\end{equation}
By \eqref{gj-7} and \eqref{gj-S130}, we have
\begin{equation}
\label{gj-13}
\begin{aligned}
E \geq \,& \sum_{i \in K}  \tilde{F}^{ii} \left\{\left(\frac{2}{1+s} - 1 - \frac{2}{\alpha}\right) \left(\frac{\nabla_i h_{11}}{h_{11}}\right)^2 + \frac{Cb^2}{\alpha}\right\}\\
  & - C \sum_{i \in J \cup \{1\}} \tilde{F}^{ii} \left\{\alpha^2 \left(\frac{\nabla_i \zeta}{\zeta}\right)^2 + b^2\right\}
\end{aligned}
\end{equation}
Assume that $\alpha$ is sufficiently large such that $\frac{2}{1+s} - 1 - \frac{2}{\alpha} > 0$. Combining \eqref{gj-12} and \eqref{gj-13}
we get
\[
0 \geq \left(b-\frac{Cb^2}{\alpha}\right) \sum \tilde{F}^{ii} + \frac{1}{2} \tilde{F}^{ii} h_{ii}^2 - \frac{C\alpha}{\zeta} - Ch_{11}
\]
provided
\[
s^2 h_{11}^2 \zeta^2 \geq C (\alpha^2 + b^2).
\]
We may further assume that $\alpha \gg b$ such that $\frac{Cb^2}{\alpha} < \frac{b}{2}$. Thus, we obtain
\begin{equation}
\label{gj-14}
0 \geq \frac{b}{2} \sum \tilde{F}^{ii} + \frac{1}{2} \tilde{F}^{ii} h_{ii}^2 - \frac{C\alpha}{\zeta} - Ch_{11}.
\end{equation}
As in \cite{CJ20}, we consider two cases, where $\epsilon_0$ is a small positive constant to be determined later.

{\bf Case 1.} \ $|h_{ii}| \leq \epsilon_0 h_{11}$ for all $i\geq2$.

Note that
\[
\eta_{ii}=\sum_{k\neq i}h_{kk}
\]
and
\[
\eta_{11}\leq \cdots \leq \eta_{nn}.
\]
Then we have
\[
[1-(n-2) \epsilon_0]h_{11} \leq \sum_{j\neq2}h_{jj} = \eta_{22}
\leq \cdots \leq \eta_{nn} \leq [1+(n-2) \epsilon_0]h_{11}.
\]
It follows that
\[
\sigma_{n-1}(\eta) \geq \eta_{22} \cdots \eta_{nn}
\geq (1- (n-1) \epsilon_0)^{n-1}h_{11}^{n-1}.
\]
Choosing $\epsilon_0$ sufficiently small we get
\[
\sigma_{n-1}(\eta) \geq \frac{h_{11}^{n-1}}{2} \geq \frac{h_{11}}{2}.
\]
We then obtain
\begin{equation}\label{Case 1 eqn 2}
\sum_{i}\tilde{F}^{ii} = \frac{n-1}{n} \psi^{\frac{1}{n}-1} \sigma_{n-1}(\eta) \geq \delta_1 h_{11}
\end{equation}
for some positive constant $\delta_1$ depending only on $n$ and $\sup \psi$.
Thus, by \eqref{gj-14} and \eqref{Case 1 eqn 2}, we have
\[
0 \geq \frac{b \delta_1}{2} h_{11} - \frac{C\alpha}{\zeta} - Ch_{11}.
\]
Fixing $b$ sufficiently large, we obtain
\[
h_{11} \zeta \leq \frac{C \alpha}{b \delta_1/2 - C}
\]
and \eqref{gj-6} is proved.

{\bf Case 2.} \ $h_{22}>\epsilon_0 h_{11}$ or $h_{nn}<-\epsilon_0 h_{11}$ .

Let
\[
\hat{F}^{ii} := \frac{\partial f^{1/n} (\lambda (h))}{\partial\eta_{ij}}.
\]
We have
\begin{equation}\label{Case 2 eqn 1}
\tilde{F}^{22} = \sum_{i\neq 2}\hat{F}^{ii} \geq \frac{1}{2}\sum_{i}\hat{F}^{ii} =
 \frac{1}{2n} \sigma_n^{1/n-1} (\eta) \sigma_{n-1}(\eta) \geq \delta_2
\end{equation}
for some positive constant $\delta_2$ depending only on $n$. Similarly, we have
\[
\tilde{F}^{nn} \geq \delta_2.
\]
Thus, by \eqref{gj-14} we have
\[
0 \geq \epsilon_0^2 \delta_2 h_{11}^2 - \frac{C\alpha}{\zeta} - Ch_{11} \geq \frac{\epsilon_0^2 \delta_2}{2} h_{11}^2 - \frac{C\alpha}{\zeta}
\]
provided $h_{11}$ is sufficiently large.
Therefore,
\[
h_{11}^2 \zeta \leq \frac{2C\alpha}{\epsilon_0^2 \delta_2}
\]
and \eqref{gj-6} follows immediately.
\end{proof}
Similarly, we can prove the following global estimates.
\begin{theorem}
\label{Thm-second-interior}
Let $u \in C^4 (\Omega) \cap C^2 (\overline{\Omega})$ be an admissible solution of \eqref{jg--1}.
Then there exists a positive constant $C$ depending on $n$, $\theta_0$
and
$\|\psi\|_{C^{2} (\overline{\Omega} \times [- \mu_0, \mu_0])}$
satisfying
\begin{equation}
\label{S-1}
\sup_{\ol \Omega} |D^2 u| \leq C \Big(1 + \sup_{\partial \Omega} |D^2 u|\Big),
\end{equation}
where $\mu_0 := \|u\|_{C^0 (\overline{\Omega})}$.
\end{theorem}


\section{Boundary estimates for second order derivatives}
In this section, we establish the boundary estimates for second order derivatives.
\begin{theorem}
\label{thm-boundary}
Suppose $\Omega$ is a bounded domain in $\mathbb{R}^n$ with smooth strictly convex boundary $\partial \Omega$.
Let $u \in C^3 (\overline{\Omega})$ be an admissible solution of \eqref{jg--1}. Then there
exists a positive constant $C$ depending only on $\theta_0$,
$\|\psi\|_{C^{1} (\overline{\Omega} \times [-\mu_0, \mu_0]) }$, $\|\varphi\|_{C^3 (\overline{\Omega})}$ and $\partial \Omega$
satisfying
\begin{equation}
\label{B2-0}
\max_{\partial \Omega} |D^2 u| \leq C,
\end{equation}
where $\mu_0 := \|u\|_{C^0 (\overline{\Omega})}$.
\end{theorem}

For any point $x_0 \in \partial \Omega$, without loss of generality, we may assume that $x_0$ is the origin and that the
positive $x_n$-axis is the inner normal direction to $\partial \Omega$ at the origin.
Furthermore, we may suppose that in a neighbourhood of the origin, the boundary $\partial \Omega$ is given by
\begin{equation}
\label{BC2-1}
x_n = \rho (x') = \frac{1}{2} \sum_{\alpha < n} \kappa^b_\alpha x_\alpha^2 + O (|x'|^3),
\end{equation}
where $\kappa^b_1, \ldots, \kappa^b_{n-1}$ are the principal curvatures of $\partial \Omega$ at the origin and $x' = (x_1, \ldots, x_{n-1})$.
Since $u = \varphi$ on $\partial \Omega$, we have
\begin{equation}
\label{BC2-2}
|u_{\alpha \beta} (0)| \leq C \mbox{  for } 1\leq \alpha, \beta \leq n - 1,
\end{equation}
where constant $C$ depending on $\|\varphi\|_{C^2 (\ol \Omega)}$.

We rewrite the equation \eqref{jg--1} by the form
\begin{equation}
\label{1-1-1}
G (D^2 u, Du) := f (\lambda (A[u])) = \psi (x, u),
\end{equation}
where $G = G (r, p)$ is viewed as a function of $(r, p)$ for $r \in S^{n \times n}$ and $p \in \mathbb{R}^n$.
Define
\begin{equation}
\label{BC2-25}
G^{ij} = \frac{\partial G}{\partial r_{ij}} (D^2 u, D u),\ \ G^{i} = \frac{\partial G}{\partial p_i} (D^2 u, D u)
\end{equation}
and the linearized operator by
\[
L = G^{ij} \partial_{ij}.
\]
Similar to lemma 2.3 of \cite{GS04}, we have the following lemma.
\begin{lemma} We have
\label{lemGS2}
\begin{equation}
\label{GS-2}
G^s = \frac{u_s}{w^2} \sum_i f_i \kappa_i + \frac{2}{w (1+w)} \sum_{t,j}F^{ij} a_{it} \big(w u_t \gamma^{sj}
   + u_j \gamma^{ts}\big),
\end{equation}
where $w = \sqrt{1 - |Du|^2}$, $a_{ij} =\frac{1}{w}\gamma^{ik}u_{kl}\gamma^{lj}$, $\kappa = \lambda (\{a_{ij}\})$, $f_i = \frac{\partial f (\kappa)}{\kappa_i}$ and
\[
F^{ij} = \frac{\partial f (\lambda (A[u]))}{\partial a_{ij}}.
\]
\end{lemma}
\begin{proof}
By straightforward calculations, we have
\begin{equation}
G^s=F^{ij}u_{kl} \frac{\partial}{\partial u_s}\big(\frac{1}{w} \gamma^{ik}\gamma^{lj}\big)
=\frac{u_s}{w^2} F^{ij} a_{ij} +\frac{2}{w} F^{ij} \gamma^{ik} u_{kl} \frac{\partial\gamma^{lj}}{\partial u_s}.
\end{equation}
From \eqref{matrix} we have
\[
\gamma^{ik} u_{kl}=w a_{ik} \gamma_{kl}.
\]
It follows that
\[
\gamma^{ik} u_{kl} \frac{\partial\gamma^{lj}}{\partial u_s}
=w a_{ik} \gamma_{kl} \frac{\partial\gamma^{lj}}{\partial u_s}
=-w a_{ik} \gamma^{lj} \frac{\partial\gamma_{kl}}{\partial u_s}
\]
since $\gamma_{kl} \gamma^{lj}=\delta_{kj}$. Next,
\begin{equation}
\frac{\partial\gamma_{kl}}{\partial u_s}
=-\frac{u_k \delta_{ls}+u_l \gamma^{ks}}{1+w}
\end{equation}
and
\begin{equation}
u_l \gamma^{lj}=\frac{u_j}{w}.
\end{equation}
Thus
\[
\gamma^{ik} u_{kl} \frac{\partial\gamma^{lj}}{\partial u_s}
=\frac{a_{ik} (w u_k \gamma^{sj}+u_j \gamma^{ks})}{1+w}.
\]
Then we obtain \eqref{GS-2}.
\end{proof}

Next, we establish the estimate
\begin{equation}
\label{BC2-3}
|u_{\alpha n} (0)| \leq C \mbox{  for } 1\leq \alpha \leq n - 1.
\end{equation}

Define
$$\omega_\delta = \{x \in \Omega: \rho (x') < x_n < \rho (x') + \delta^2 , |x'| < \delta\},$$
we can find that the boundary $\partial \omega_\delta$ consists of three parts:
$$\partial \omega_\delta
= \partial_1 \omega_\delta \cup \partial_2 \omega_\delta \cup \partial_3 \omega_\delta,$$ where
$\partial_1 \omega_\delta$, $\partial_2 \omega_\delta$ and $\partial_3 \omega_\delta$  are defined by $\{x_n=\rho\} \cap\overline{\omega}_{\delta}$, $\{ x_n=\rho+\delta^2\}\cap\overline{\omega}_{\delta}$
and $\{|x'| = \delta\}\cap\overline{\omega}_{\delta}$ respectively.
Since $\Omega$ is admissible,
there exist two positive constants $\theta$ and $K$ satisfying
\begin{equation}
\label{BC2-5}
(\kappa_1^b - 3 \theta, \ldots, \kappa_{n-1}^b - 3 \theta, 2 K) \in \Gamma.
\end{equation}
Define
\begin{equation}
\label{BC2-6}
v = \rho (x') - x_n - \theta |x'|^2 + K x_n^2.
\end{equation}
We see that when $\delta$ depending on $\theta$ and $K$ is sufficiently small, we have
\begin{equation}
\label{BC2-12}
\begin{aligned}
v \leq & - \frac{\theta}{2} |x'|^2, & \mbox{ on } \partial_1 \omega_\delta\\
v \leq & - \frac{\delta^2}{2}, & \mbox{ on } \partial_2 \omega_\delta\\
v \leq & - \frac{\theta \delta^2}{2},   & \mbox{ on } \partial_3 \omega_\delta.
\end{aligned}
\end{equation}
In view of \eqref{BC2-1} and \eqref{BC2-5},  $\lambda (D^2 v)\in \Gamma$
on $\overline{\omega}_\delta$.
Thus, there exists an uniform constant $\eta_0 > 0$ depending only on $\theta$, $\p\Omega$
and $K$ satisfying
\[
\lambda (D^2 v - 2 \eta_0 I) \in \Gamma \mbox{ on } \overline{\omega}_\delta.
\]
Then we have
\begin{equation}
\label{BC2-4}
\lambda \left(\frac{1}{w} \{\gamma^{is} (v_{st} - 2 \eta_0 \delta_{st}) \gamma^{jt}\}\right) \in \Gamma
\mbox{ on }
\overline{\omega}_\delta.
\end{equation}
To prove \eqref{BC2-3}, we shall use the strategy of \cite{Ivochkina90} to consider
the function
\[
W := \nabla'_\alpha (u - \varphi) - \frac{1}{2} \sum_{1\leq \beta \leq n - 1} (u_\beta - \varphi_{\beta})^2
\]
defined on $\ol \omega_\delta$ for small $\delta$,
where
\[
\nabla'_\alpha u := u_\alpha + \rho_\alpha u_n, \mbox{ for } 1\leq \alpha \leq n - 1.
\]
Since the proof of the following lemma is similar to that of Lemma 5.3 in \cite{JaoSun22}, we omit its proof. For reader's convenience, we provide
a detailed proof for a similar result (Lemma \ref{gj-lem2}) later.
\begin{lemma} If $\delta$ is sufficiently small, we have
\label{BC2-lem1}
\begin{equation}
\label{BC2-15}
LW \leq C \left(1 + |D W| + \sum_i G^{ii} + G^{ij} W_i W_j\right),
\end{equation}
where $C$ is a positive constant depending on $n$, $\theta_0$, $\|\psi\|_{C^1 (\ol \Omega \times [-\mu_0, \mu_0]}$, $\|\varphi\|_{C^3(\ol \Omega)}$ and $\partial \Omega$, where $\mu_0 = \|u\|_{C^0 (\overline{\Omega})}$.
\end{lemma}

As \cite{JW21} and \cite{JaoSun22}, we consider the following barrier on $\overline{\omega}_\delta$, for sufficiently small $\delta$,
\begin{equation}
\label{BC2-20}
\Psi := v - td + \frac{N}{2} d^2,
\end{equation}
where $v(x)$ is defined by \eqref{BC2-6}, $d (x) := \mathrm{dist} (x, \partial \Omega)$ is the distance from $x$ to the boundary $\partial \Omega$,
and $t,N$ are two positive constants to be determined later. Since $f^{1/n}$ is concave in $\Gamma$ and homogeneous of degree one and $|Dd| \equiv 1$ on the boundary $\partial \Omega$,  by \eqref{BC2-4} and \eqref{cj-2}, we have,
\[
\begin{aligned}
 &\frac{1}{n} \psi^{\frac{1}{n} - 1}G^{ij} (D^2 v - \eta_0 I + N D d \otimes D d)_{ij}\\
   \geq \,& G^{1/n} (D^2 v - \eta_0 I + N D d \otimes D d, Du)\\
   \geq \,& \mu (N) \mbox{ on } \overline{\omega}_\delta
\end{aligned}
\]
for some positive constant $\mu (N)$ satisfying $\lim_{N \rightarrow + \infty}\mu (N) = +\infty$.
We then have
\begin{equation}
\label{gj-15}
\begin{aligned}
G^{ij} \Psi_{ij} \geq \,& n \psi^{1-1/n}\mu(N) + \eta_0 \sum_i G^{ii} + (N d - t) G^{ij} d_{ij}\\
\geq \,& n \epsilon_0^{1-1/n}\mu(N) + \eta_0 \sum_i G^{ii} + (N d - t) G^{ij} d_{ij}\\
  \geq \,& 2 \mu_1(N) + (\eta_0 - C N \delta - C t) \sum_i G^{ii}
\end{aligned}
\end{equation}
on $\overline{\omega}_\delta$, where $\mu_1 (N) := n (\inf \psi)^{1-1/n}\mu(N)/2$.
Define
\begin{equation}\label{new3.2}
\tilde{W} := 1 - \exp\{- b W\}.
\end{equation}
By \eqref{BC2-15}, we can choose the constant $b$ large enough so that
\begin{equation}
	\label{new3.4}
	\begin{aligned}
		L \tilde{W} = \,& G^{ij} \big(-e^{-bW} b^2 W_i W_j + b e^{-bW} W_{ij}\big)\\
		\leq \,& b e^{-bW} \left[ C \left(1 + |D W| + \sum_i G^{ii} \right) + (C - b) G^{ij} W_i W_j\right]\\
		\leq \,&  C (1 + |D \tilde{W}| + \sum_i G^{ii}) + (C - b) G_{ij} W_i W_j b e^{-bW}\\
		\leq \,&  C (1 + |D \tilde{W}| + \sum_i G^{ii}).
	\end{aligned}
\end{equation}
We consider the function
\[
\Phi := R \Psi - \tilde{W},
\]
where $R$ is a large undetermined positive constant. We shall prove
\begin{equation}
\label{add-6}
\Phi \leq 0 \mbox{ on } \overline{\omega}_\delta
\end{equation}
by choosing suitable positive constants $\delta$, $t$, $N$ and $R$.

We first consider the case that the maximum of $\Phi$
is achieved at an interior point $x_0 \in \omega_\delta$. It follows that at $x_0$,
\[
|D\tilde{W}|=R |D\Psi|
\]
and if $N$ is sufficiently large and $\delta < \sqrt{\mu_1 (N)}/2CN$,
\[
|D \Psi| = |D v - t D d + N d D d| \leq C (1+t) + C\delta N \leq \mu_1(N)^{1/2} \mbox{ in } \omega_\delta.
\]
Therefore, by \eqref{gj-15}, provided $\delta$ and $t$ sufficiently small such that $C N \delta + C t < \eta_0/2$, we have
\begin{equation}
\label{BC2-21}
L \Psi \geq \mu_1 (N) + \mu_1(N)^{1/2} |D \Psi| + \frac{\eta_0}{2} \sum_i G^{ii}.
\end{equation}
By \eqref{new3.4} and \eqref{BC2-21} we obtain, at $x_0$,
\[
\begin{aligned}
0 \geq L \Phi
  \geq \,& R\mu_1 (N) + R\mu_1(N)^{1/2} |D \Psi| + \frac{R\eta_0}{2} \sum G^{ii}\\
     & - C \left(1 + |D \tilde{W}| + \sum G^{ii}\right)\\
 \geq \,& R \mu_1(N) - C + R (\mu_1(N)^{1/2} - C)|D \Psi|  \\\,&+\left(\frac{R\eta_0}{2} - C\right)\sum G^{ii} > 0
\end{aligned}
\]
provided $N$ and $R$ are chosen sufficiently large which is a contradiction. Thus, the maximum of $\Phi$
is achieved at the boundary $\partial \omega_\delta$.
We may further assume $\delta < 2t/N$ so that
\begin{equation}\label{new3.1}
- t d + \frac{N}{2} d^2 \leq 0 \mbox{ on } \overline{\omega}_\delta.
\end{equation}
By \eqref{BC2-12} and \eqref{new3.1},
we can conclude that $\Phi \leq 0$ on $\partial \omega_\delta$ by choosing $R$ larger and then \eqref{add-6} is proved.

Since $(R \Psi - \tilde{W}) (0) = 0$, we have $(R \Psi - \tilde{W})_n (0) \leq 0$. Therefore, we get
\[
u_{n \alpha} (0) \geq - C.
\]
The above arguments also hold for
\[
W = - \nabla'_\alpha u - \frac{1}{2} \sum_{1 \leq \beta \leq n - 1} (u_\beta-\varphi_\beta)^2.
\]
Hence, we obtain \eqref{BC2-3}.

Since the mean curvature of $M_u$, $H > 0$, it suffices to prove an upper bound
\begin{equation}
\label{gj-16}
u_{nn} (0) \leq C.
\end{equation}
We use an idea of \cite{Ivochkina91} and \cite{ILT96}
to prove
\begin{equation}
		\label{normal}
		S_{1;n} (D^2 u, Du)
				= \sigma_{1} \left( \lambda \left(I + \frac{Du(n) \otimes Du(n)}{1 - |Du(n)|^2}\right) D^2 u(n) \right)
		\geq c_0
\end{equation}
for some uniform constant $c_0 >0$. Then \eqref{gj-16} follows immediately by \eqref{jg-4} if \eqref{normal} is proved.

To prove \eqref{normal}, we introduce some notations.
Suppose  $W$  is a $(0,2)$ tensor field on $\overline{\Omega}$, namely $W\in C^2(T^*\overline{\Omega}\otimes T^*\overline{\Omega})$,
where $T^*\overline{\Omega}$ is the co-tangent bundle of $\overline{\Omega}$.
Let $W'$
be the projection of $W|_{\partial\Omega}$ in the bundle $T^* \partial \Omega\otimes T^*\partial\Omega$,
where $T^*\partial \Omega$ is the co-tangent bundle of $\partial \Omega$. Similarly, for a $1$-form $P$ on $\ol \Omega$,
denote $P'$ to be the projection of $P|_{\partial \Omega}$ in the co-tangent bundle $T^* \partial \Omega$.
Denote $\tilde{g}'$ to be the induced metric on $\partial \Omega \subset \ol \Omega$. To proceed we recall an easy lemma as follows.
\begin{lemma}
\label{gj-lem1}
Let $V'$, $W'$ be $(0,2)$ tensor fields on $\partial \Omega$ defined by
\[
V' = V'_{\alpha\beta} \theta^\alpha \otimes \theta^\beta;\ \ W' = W'_{\alpha\beta} \theta^\alpha \otimes \theta^\beta
\]
locally, where $\{\theta^1, \ldots, \theta^{n-1}\}$ is an arbitrary local frame for $T^* \partial \Omega$. Suppose
\[
\tilde{g}' = \tilde{g}'_{\alpha\beta} \theta^\alpha \otimes \theta^\beta
\]
and $\{\tilde{g}'^{\alpha\beta}\}_{1\leq \alpha, \beta \leq n-1}$ is the inverse of $\{\tilde{g}'_{\alpha\beta}\}_{1\leq \alpha, \beta \leq n-1}$.
Define
\[
\{U'_{\alpha \beta}\} = \{V'_{\alpha\beta}\} \{\tilde{g}'^{\alpha\beta}\} \{W'_{\alpha\beta}\},
\]
namely
\[
U'_{\alpha \beta} = \sum_{1\leq\gamma_1, \gamma_2\leq n-1}V'_{\alpha\gamma_1} g'^{\gamma_1\gamma_2} W'_{\gamma_2\beta}, 1\leq \alpha, \beta \leq n-1.
\]
Then
\[
U' := U'_{\alpha \beta} \theta^\alpha \otimes \theta^\beta
\]
is a $(0,2)$ tensor field on $\partial \Omega$. For simplicity, we write $U' = V' W'$ in the following.
\end{lemma}
We set
\[
m:= \inf_{\partial{\Omega}} \mathrm{tr}_{\tilde{g}'} \left(\left(\tilde{g}' + \frac{(Du)'\otimes(Du)'}{1-|(Du)'|_g^2}\right) (D^2 u)'\right),
\]
where $\mathrm{tr}_{\tilde{g}'}$ denotes the trace with respect to the metric $\tilde{g}'$.
We shall prove $m \geq c_0$ for some positive constant $c_0$.
Suppose $m$ is attained at $x_0 \in \partial \Omega$. We may assume that $x_0$ is the origin and the
positive $x_n$-axis is in the interior normal direction to $\partial \Omega$ at the origin as before.
Furthermore, we may also assume the boundary $\partial \Omega$ is given by \eqref{BC2-1} near the origin.

Choose a local orthonormal frame
$\{\tilde{e}_1, \ldots, \tilde{e}_n\}$ around $x_0$ such that $e_n$ is the interior normal to $\partial \Omega$.
$\tilde{\nabla}$ denotes the standard connection of $\mathbb{R}^n$. Write $\tilde{e}_i (x) = \tilde{e}^j_i (x)\partial_j$ for $i = 1, \ldots n$, where $\partial_1,\cdots,\partial_n$ is the  rectangular coordinate system. Thus, we have
 \[
\tilde{\nabla}_i u:=\tilde{\nabla}_{\tilde{e}_i}u= \tilde{e}_i^j \partial_ju=\tilde{e}_i^ju_j
\]
and
\[
\tilde{\nabla}_{ij} u:=\tilde{\nabla}_{\tilde{e}_i}\nabla_{\tilde{e}_j}u =\tilde{e}_i^k\tilde{e}_j^l\partial_{k}\partial_l u = \tilde{e}_i^k \tilde{e}_j^l u_{kl}.
\]
We may also assume that $\tilde{e}_i^j (x_0) = \delta_{ij}$
for $1 \leq i, j \leq n$ and $\{\tilde{\nabla}_{\alpha \beta} u (x_0)\}_{1 \leq \alpha, \beta \leq n-1}$ is diagonal.

Since $u - \varphi = 0$ on $\partial \Omega$, we find
\begin{equation}
\label{cor-2}
\tilde{\nabla}_{\alpha \beta} u = \tilde{\nabla}_{\alpha \beta} \varphi - \tilde{\nabla}_n (u - \varphi) \sigma_{\alpha \beta}, \ 1 \leq \alpha, \beta \leq n - 1
\end{equation}
on $\partial \Omega$ near $x_0$, where $\sigma_{\alpha \beta} = \langle D_{e_\alpha} e_\beta, e_n\rangle$. Note that $\sigma_{\alpha \beta}$
is the second fundamental form of $\partial \Omega$ when it is restricted on $\partial \Omega$.
By \eqref{cor-2} and the definition of $m$, we have
\begin{equation}
\label{gj-17}
\begin{aligned}
- \eta := \,&  - A (x) \tilde{\nabla}_n (u - \varphi)
   = \left(\delta_{\alpha\beta}+\frac{\tilde{\nabla}_\alpha \varphi\tilde{\nabla}_\beta \varphi}{1-|\tilde{\nabla}' \varphi|^2}\right)
     \tilde{\nabla}_{\alpha \beta} (u - \varphi)\\
     \geq \,& m - \left(\delta_{\alpha\beta}+\frac{\tilde{\nabla}_\alpha \varphi\tilde{\nabla}_\beta \varphi}{1-|\tilde{\nabla}' \varphi|^2}\right)
     \tilde{\nabla}_{\alpha \beta} \varphi
\end{aligned}
\end{equation}
on $\partial \Omega$ near the origin,
where
\[
A (x) := \sum_{1\leq\alpha, \beta \leq n-1} \sigma_{\alpha \beta} \left(\delta_{\alpha\beta}+\frac{\tilde{\nabla}_\alpha \varphi\tilde{\nabla}_\beta \varphi}{1-|\tilde{\nabla}' \varphi|^2}\right)
\]
and $|\tilde{\nabla}' \varphi|^2 := \sum_{\alpha \leq n-1} (\tilde{\nabla}_\alpha \varphi)^2$.
Define
\[
V = - \eta(x) - m + \left(\delta_{\alpha\beta}+\frac{\tilde{\nabla}_\alpha \varphi\tilde{\nabla}_\beta \varphi}{1-|\tilde{\nabla}' \varphi|^2}\right)
     \tilde{\nabla}_{\alpha \beta} \varphi - \frac{K}{2}\sum_{\beta = 1}^{n-1} (u - \varphi)_\beta^2
\]
in
$\omega_\delta = \{x \in \Omega: \rho (x') < x_n < \rho (x') + \delta^2 , |x'| < \delta\}$.
\begin{lemma}
\label{gj-lem2}
There exist positive constants $\delta$ sufficiently small and $K$ sufficiently large such that
\begin{equation}
\label{gj-18}
LV \leq C \left(1 + |D V| + \sum_i G^{ii} + G^{ij} V_i V_j\right),
\end{equation}
where $C$ is a positive constant depending on $n$, $\theta_0$, $\|\psi\|_{C^1 (\ol \Omega \times [-\mu_0, \mu_0]}$, $\|\varphi\|_{C^3(\ol \Omega)}$ and $\partial \Omega$, where $\mu_0 = \|u\|_{C^0 (\overline{\Omega})}$.
\end{lemma}
\begin{proof}
Differentiating the equation \eqref{1-1-1}, we get
\begin{equation}
\label{BC2-16}
\begin{aligned}
L V + G^s V_s
  \leq \,& C K + C \sum_i G^{ii} + C \sum_s |G^{s}| - 2 G^{ij} u_{ti} \left(Ae_n^t\right)_j   \\
     & + 2 K \sum_{\beta \leq n-1} G^{ij} u_{\beta i} \varphi_{\beta j} - K \sum_{\beta \leq n-1} G^{ij} u_{\beta i} u_{\beta j}.
\end{aligned}
\end{equation}
Since
\[
a_{ij} = \frac{1}{w}\gamma^{ik}u_{kl} \gamma^{lj},
\]
we have
\[
G^{ij} = \frac{\partial G}{\partial u_{ij}} = F^{st}\frac{\partial a_{st}}{\partial u_{ij}} = \frac{1}{w}\sum_{s,t} F^{st} \gamma^{is} \gamma^{tj}
\]
and
\[
u_{ij} = w\sum_{s,t} \gamma_{is} a_{st} \gamma_{tj}.
\]
Here $F^{ij}$ is defined by
\[
F^{ij} := \frac{\partial f(\lambda (A[u]))}{\partial a_{ij}}.
\]
It follows that
\[
\sum_{\beta \leq n - 1} G^{ij} u_{\beta i} u_{\beta j}
  = w \sum_{\beta \leq n - 1} \sum_{s,t}F^{ij} \gamma_{\beta s} \gamma_{\beta t} a_{si} a_{tj}
\]
and
\[
C_0^{-1} \sum_{i} F^{ii}\leq\sum_{i}G^{ii}\leq C_0 \sum_{i} F^{ii}
\]
for some positive constant $C_0$ depending on $\theta_0$.
We can find an orthogonal matrix $B = \{b_{ij}\}$ which can diagonalize $\{a_{ij}\}$ and $\{F^{ij}\}$ simultaneously, i.e.,
\[
F^{ij} =\sum_s b_{is} f_s b_{js} \mbox{ and } a_{ij} =\sum_s b_{is} \kappa_s b_{js}.
\]
Therefore, we have
\begin{equation}
	\begin{aligned}
		\sum_{\beta \leq n - 1} G^{ij} u_{\beta i} u_{\beta j} =& w \sum_{\beta \leq n - 1} F^{ij} \gamma_{\beta s} \gamma_{\beta t} a_{si} a_{tj}\\\notag	
		=& w \sum_{\beta \leq n - 1} \sum_i\left(\sum_s\gamma_{\beta s} b_{si}\right)^2  f_i \kappa_i^2\notag.
	\end{aligned}
\end{equation}
Let $\eta = (\eta_{ij}) = (\sum_s\gamma_{is} b_{sj})$. We find $\eta \cdot \eta^T = g$ and $|\det (\eta)| = \sqrt{1 - |Du|^2}$.
Hence, we obtain
\begin{equation}
\label{BC2-14}
\sum_{\beta \leq n - 1} G^{ij} u_{\beta i} u_{\beta j}
  = w \sum_{\beta \leq n - 1} \sum_i\eta_{\beta i}^2  f_i \kappa_i^2.
\end{equation}
We have
\begin{equation}
\label{BC2-19}
\left|G^{ij} u_{ti} \left(Ae_n^t\right)_j\right| = \bigg|\sum_{i,t}f_i \kappa_i b_{si} \gamma^{js} b_{pi} \gamma_{tp} \left(Ae_n^t\right)_j\bigg| \leq C \sum_i f_i |\kappa_i|.
\end{equation}
and
\begin{equation}
	\label{BC2-19.1}
	\left|G^{ij} u_{\beta i} \varphi_{\beta j}\right|= \bigg|\sum_{i,t}f_i \kappa_i b_{si} \gamma^{js} b_{ti} \gamma_{ \beta t} \varphi_{\beta j}\bigg| \leq C \sum_i f_i |\kappa_i|.
\end{equation}
For any indices  $j, t$, we have
\[
F^{ij} a_{it} =\sum_{i,s,p} b_{is} f_s b_{js} b_{ip} \kappa_p b_{tp} = \sum_i f_i \kappa_i b_{ji} b_{ti}.
\]
Thus,
by \eqref{GS-2}, we find
\begin{equation}
\label{BC2-11}
\begin{aligned}
\mid \sum_s G^s \mid =& \bigg| \sum_s \biggl\{ \frac{u_s}{w^2} \sum_i f_i \kappa_i + \frac{2}{w (1+w)} \sum_{t,j}F^{ij} a_{it} \big(w u_t \gamma^{sj}
+ u_j \gamma^{ts}\big) \biggr\} \bigg|\\\notag
\leq&  C \sum_i f_i |\kappa_i|.
\end{aligned}
\end{equation}
Combining \eqref{BC2-16}-\eqref{BC2-11},
we obtain
\begin{equation}
\label{BC2-26}
\begin{aligned}
LV + G^s V_s
     \leq C K \left(1 + \sum_i G^{ii} + \sum_i f_i |\kappa_i|\right) - w K\sum_{\beta \leq n - 1} \sum_i\eta_{\beta i}^2 f_i \kappa_i^2.
\end{aligned}
\end{equation}
Now we consider the term $G^s V_s$. We have, by \eqref{GS-2} and the definition of the matrix $\{b_{ij}\}$,
\begin{equation}
\label{BC2-17}
\begin{aligned}
- G^s V_s = \,& -\frac{u_s}{w^2} \sum_i f_i \kappa_i V_s - \frac{2}{w (1+w)} \sum_{t,j}F^{ij} a_{it} \big(w u_t \gamma^{sj}
+ u_j \gamma^{ts}\big) V_s\\
   = \,& -\frac{1}{w} \sum_s\left(n \psi u_s\frac{1}{w} + 2 \sum_{t,i}f_i \kappa_i (b_{ti} u_t) \gamma^{sl} b_{li}\right) V_s\\
   \leq \,& C |D V| - \frac{2}{w} \sum_{t,i}f_i \kappa_i (b_{ti} u_t) \gamma^{sl} b_{li} V_s.
\end{aligned}
\end{equation}
We consider two cases: (a) $\sum_{\beta \leq n - 1} \eta_{\beta i}^2 \geq \epsilon$ for all $i$;
and (b) $\sum_{\beta \leq n - 1} \eta_{\beta r}^2 < \epsilon$ for some index $1 \leq r \leq n$, where
$\epsilon$ is a positive constant to be chosen later.

For the case (a), by \eqref{BC2-14}, we have
\[
\sum_{\beta \leq n - 1} G^{ij} u_{\beta i} u_{\beta j}
  \geq \epsilon w \sum_{i} f_i \kappa_i^2.
\]
By Cauchy-Schwarz inequality, for any $\epsilon_0 > 0$, we have
\begin{equation}
\label{BC2-31}
\frac{2}{w} \kappa_i (b_{ti} u_t) \gamma^{sl} b_{li} V_s
  \geq -\frac{\epsilon_0}{2} \kappa_i^2 - \frac{C}{\epsilon_0 } (\gamma^{sl} b_{li} V_s)^2.
\end{equation}
Then
\[
- G^s V_s \leq C |D V| + \frac{\epsilon_0}{2} f_i \kappa_i^2 + \frac{C}{\epsilon_0} G^{ij} V_i V_j.
\]
Obviously, for any $\epsilon_1 > 0$,
$$\sum_i f_i |\kappa_i|\leq \frac{1}{2\epsilon_1} \sum_i f_i +  \frac{\epsilon_1}{2} \sum_i f_i \kappa_i^2
   \leq C\left(\frac{1}{\epsilon_1} \sum_i G^{ii}+ \epsilon_1 \sum_i f_i \kappa_i^2\right).$$
 Combining the previous four inequalities with \eqref{BC2-26}, \eqref{gj-18} follows.

For the case (b),
by lemma 4.3 of Bayard \cite{Bayard03} ,
for any $i\neq r$,
$$\sum_{\beta\leq n-1}\eta^2_{\beta i}\geq c_1$$ for some positive constant $c_1$ depending on $\theta_0$ and $n$.

In view of \eqref{BC2-14}, It follows that
\begin{equation}
\label{BC2-13}
\sum_{\beta \leq n - 1} G^{ij} u_{\beta i} u_{\beta j} \geq w c_1 \sum_{i \neq r} f_i  \kappa_i^2.
\end{equation}
If $\kappa_r \leq 0$, by Lemma 2.7 of \cite{Guan14}, we have
\[
 \sum_{i \neq r} f_i  \kappa_i^2
    \geq \frac{1}{n+1} \sum_{i = 1}^n f_i \kappa_i^2.
\]
Thus \eqref{gj-18} follows using a similar argument as the Case (a).

We need only to deal with the case $\kappa_r > 0$. Without loss of generality, we assume $r = 1$.
Now we consider two cases.

\noindent
{\bf Case (b-1).} \ $|\kappa_i| \leq \epsilon_0 \kappa_1$ for all $i\geq2$, where the positive constant $\epsilon_0$ is some positive constant to be determined.

In this case, as in Section 4, we derive,
\[
(1 - (n-2)\epsilon_0) \kappa_1 \leq \lambda_i \leq (1 + (n-2)\epsilon_0) \kappa_1, \mbox{ for } i \geq 2,
\]
where $\lambda_1, \ldots, \lambda_n$ are defined in \eqref{lambda}.
By the equation \eqref{jg--1}, we get,
\[
\lambda_1 = \frac{\psi}{\lambda_2 \cdots \lambda_n} \leq C \kappa_1^{1-n}
\]
by fixing the constant $\epsilon_0$ sufficiently small. It follows that
\[
\sigma_{n-1;i} (\lambda) = \prod_{j\neq i} \lambda_j \leq C \kappa_1^{n-2} \kappa_1^{1-n} = C \kappa_1^{-1}
   \mbox{ for } i \geq 2.
\]
and hence
\[
f_1 (\kappa) = \sum_{i \neq 1} \sigma_{n-1; i} (\lambda) \leq C \kappa_1^{-1}.
\]
Therefore,
\[
f_1 \kappa_1 \leq C.
\]
By \eqref{BC2-17} and \eqref{BC2-31} and the Cauchy-Schwarz inequality, we have, for any $\epsilon > 0$,
\begin{equation}
\label{BC2-30}
\begin{aligned}
- G^s  V_s \leq \,& C |DV| - \frac{2}{w} \sum_{i \neq 1}f_i \kappa_i (b_{ti} u_t) \gamma^{sl} b_{li} V_s\\
  \leq \,& C |DV| + C \sum_{i\neq 1} f_i |\kappa_i| |\gamma^{sl} b_{li} V_s| \\
  \leq \,& C |DV| + \epsilon \sum_{i \neq 1} f_i \kappa_i^2
     + \frac{C}{\epsilon} \sum_{i=1}^n f_i \gamma^{sl} b_{li} V_s \gamma^{tk} b_{ki} V_t\\
  \leq \,& C |DV| + \epsilon \sum_{i \neq 1} f_i \kappa_i^2 + \frac{C}{\epsilon} G^{ij} V_i V_j.
\end{aligned}
\end{equation}
By using \eqref{BC2-13} and  fixing $\epsilon$ sufficiently small, we can prove \eqref{BC2-15}.

\noindent
{\bf Case (b-2).} \ $|\kappa_{i_0}| > \epsilon_0 \kappa_1$ for some $i_0 \geq2$,

First by direct calculation, we have
\begin{equation}
\label{BC2-22}
\begin{aligned}
\left|\gamma^{sl} b_{l1} V_s\right| \leq \,& C + \left|\gamma^{sl} b_{l1} \Big((-Ae^p_n)u_{p s}
   - K\sum_{\beta \leq n-1} (u-\varphi)_\beta (u-\varphi)_{\beta s}\Big)\right|\\
     \leq \,& CK + w \left|(-Ae^p_n) \eta_{p 1} - K \sum_{\beta \leq n-1} (u-\varphi)_\beta \eta_{\beta 1}
   \right| \kappa_1\\
  \leq \,& C w (\epsilon K + 1) \kappa_1 + CK.
\end{aligned}
\end{equation}
Note that
\[
f_1 \kappa_1 = n \psi - \sum_{i \neq 1} f_i \kappa_i.
\]
We have
\[
\begin{aligned}
 \frac{2}{w} f_1 \kappa_1 \bigg|\left(\sum_tb_{t1} u_t\right) & \gamma^{sl} b_{l1} V_s\bigg| = \frac{2}{w} \big(n\psi - \sum_{i \neq 1} f_i \kappa_i\big)
    \bigg|\left(\sum_tb_{t1} u_t\right) \gamma^{sl} b_{l1} V_s\bigg| \\
 \leq \,& C |D V| + C (\epsilon K + 1) \sum_{i \neq 1} f_i |\kappa_i| \kappa_1 + C K\sum_{i\neq 1} f_i |\kappa_i|\\
\leq \,& C |D V| + \left(\epsilon_0^{-1} C (\epsilon K + 1) + \epsilon_1K\right) \sum_{i\neq 1} f_i \kappa_i^2
   + \frac{CK}{\epsilon_1} \sum_{i\neq 1} f_i
\end{aligned}
\]
for any $\epsilon_1 > 0$.
We can choose sufficiently small $\delta$, $\epsilon$ and $\epsilon_1$ and sufficiently large $K$ satisfying
\[
\epsilon_0^{-1} C (\epsilon K + 1) + \epsilon_1K < \frac{m c_1K}{4},
\]
where $m=1-\theta_0^2$ and $\theta_0$ is defined in \eqref{spacelike}.
Therefore, as in Case (b-1), \eqref{gj-18} follows.
\end{proof}
Using same arguments as in the proof of \eqref{add-6}, we can choose positive constant $R$ sufficiently large such that
$R \Psi - \tilde{V} \leq 0$ on $\bar{\omega}_\delta$,
where
\[
\tilde{V} := 1 - \exp\{- b V\}
\]
as in \eqref{new3.2}. Since $\Omega$ is convex and admissible, there exists a positive constant $a_0$ depending
only on the geometry of $\partial \Omega$ such that
\[
A (x) \geq \sum_{1 \leq \alpha \leq n-1} \sigma_{\alpha \alpha} \geq a_0 \mbox{ for all } x \in \bar{\omega}_\delta.
\]
Then we obtain an upper bound
\[
u_{nn} (x_0) \leq C.
\]
Thus, the principle curvatures $\kappa [M_u(x_0)]$ lie in \emph{a priori} compact subset $S \Subset \Gamma$.
By \eqref{ILT}, there exists a positive constant $a_1$ depending only on $S$ such that
\[
\begin{aligned}
a_1 \leq \frac{\partial S_2}{\partial u_{nn}} (D^2 u (x_0), D u (x_0)) = \,& \frac{1 - |Du(n) (x_0)|^2}{1-|Du(x_0)|^2} S_{k-1; n} (D^2 u (x_0), D u (x_0))\\
   = \,& \frac{1 - |Du(n) (x_0)|^2}{1-|Du(x_0)|^2} m.
\end{aligned}
\]
It follows that $m \geq c_0$ for some positive constant depending only on $S$ and $\theta_0$. Then \eqref{normal} and \eqref{gj-16}
follow immediately.
Theorem \ref{thm-boundary} is proved.

\bigskip


\end{document}